\documentclass[reqno]{amsart}
\usepackage{amsmath}
\usepackage{amssymb}
\usepackage{amsthm}
\usepackage{mathrsfs}
\usepackage{url}
\newtheorem{definition}{Definition}[section]
\newtheorem{theorem}[definition]{Theorem}

\newtheorem{cor}[definition]{Corollary}
\newtheorem{question}[definition]{Question}
\newtheorem{example}[definition]{Example}
\newtheorem{remark}[definition]{Remark}

\usepackage{color}
\newcommand{\V}{\mathcal{V}}
\begin{document}

\title{A refined nc Oka-Weil theorem}
\author{Kenta Kojin}
\address{
Graduate School of Mathematics, Nagoya University, 
Furocho, Chikusaku, Nagoya, 464-8602, Japan
}
\email{m20016y@math.nagoya-u.ac.jp}
\date{\today}

\maketitle
\begin{abstract}
This short note refines nc Oka-Weil theorem by using a characterization of free compact nc sets based on the notion of dilation hulls. A consequence of it is that any free holomorphic function can be represented as a free polynomial on each free compact nc set.
\end{abstract}


\section{Introduction and Preliminaries}

The celebrated Oka-Weil theorem in the classical complex analysis implies that any holomorphic function can uniformly be approximated by polynomials on each polynomially convex compact set (See e.g., \cite[Chap. III, Theorem 5.1]{Gam}). Thus, the theorem is quite fundamental in applications. Agler and McCarthy  \cite{AMg} proved its nc analog (nc stands for noncommutative), and then Ball, Marx and Vinnikov  \cite{BMV} proved it in more general settings. The goal of this note is to refine it. 

We review some preliminary materials on nc functions. Let $\V$ be a complex vector space. Thus $\V$ can be regarded as a bimodule over $\mathbb{C}$. Denote by $\V^{m\times n}$ all the $m\times n$ matrices over $\V$. We define the associated nc space $\V_{nc}$ as the disjoint union of $\V^{n\times n}$ all over $n\in\mathbb{N}$:

\begin{equation*}
\V_{nc}=\displaystyle\coprod_{n=1}^{\infty}\V^{n\times n}.
\end{equation*}
A subset $\Omega$ of $\V_{nc}$ is said to be an {\bf nc set} if $\Omega$ is closed under direct sums as follows. If $x\in\Omega_n:=\Omega\cap\V^{n\times n}$ and $y\in\Omega_m$, then $x\oplus y=
\begin{bmatrix}
x&0\\
0&y
\end{bmatrix}\in\Omega_{n+m}$.

Assume next that $\Omega$ is an nc set of $\V_{nc}$ and $\V_0$ is another complex vector space. For $\alpha\in\mathbb{C}^{n\times m}$, $x\in\V^{m\times k}$, $\beta\in\mathbb{C}^{k\times l}$, the bimodule structure of $\V$ over $\mathbb{C}$ enables us to define the matrix multiplication $\alpha x \beta\in\V^{n\times l}$ naturally, and similarly $\alpha V\beta$ makes sense as an element of $\V_0^{n\times l}$ for $V\in\V_0^{m\times k}$. A function $f:\Omega\rightarrow\V_{0,nc}$ is an {\bf nc function} if
\begin{enumerate}
\item $f$ is {\bf graded}, i.e., if $x\in\Omega_n$, then $f(x)\in\V_0^{n\times n}$, and
\item $f$ {\bf respects intertwinings}, i.e., whenever $x\in\Omega_n$, $y\in\Omega_m$ and $\alpha\in\mathbb{C}^{m\times n}$ satisfy $\alpha x=y\alpha$, then $\alpha f(x)=f(y)\alpha$.
\end{enumerate}
Note that $f$ is an nc function if and only if it satisfies the following conditions (see \cite[Section I.2.3]{KVV}):
\begin{enumerate}
\item $f$ is graded,
\item $f$ {\bf respects direct sums}, i.e., if $x$, $y$ and $x\oplus y$ are in $\Omega$, then $f(x\oplus y)=f(x)\oplus f(y)$, and
\item $f$ {\bf respects similarities}, i.e., whenever $x, y\in\Omega_n$,
 $\alpha\in\mathbb{C}^{n\times n}$ with $\alpha$ invertible such that
  $y=\alpha x\alpha^{-1}$, then $f(y)=\alpha f(x)\alpha^{-1}$.
\end{enumerate}

Next, we define a free topology. A subset $\Xi$ of $\V_{nc}$ is a {\bf full nc subset }of $\V_{nc}$ if $\Xi$ is an nc set and {\bf invariant under left injective intertwinings}, that is, whenever $x\in\Xi$, $y\in\V^{m\times m}$ such that $Iy=xI$ for some injective $I\in\mathbb{C}^{n\times m}$, 
then $y\in\Xi_m$. We assume that $\mathbf{A}$ is an algebra of nc functions from a full nc set $\Xi$ of $\V_{nc}$ into $\mathbb{C}_{nc}$. If $Q=[q_{ij}]$ is a finite-size matrix with entries 
$q_{ij}$ in $\mathbf{A}$, we define the nc subset $\mathbb{D}_Q\subset\Xi$ by
\begin{equation*}
\mathbb{D}_Q=\{z\in\Xi\;|\;\|Q(z)\|<1\}.
\end{equation*}
Here $\|Q(z)\|$ denotes the operator norm. The nc subset $\mathbb{D}_Q$ is called a {\bf basic A-free open set}. Since the intersection of two basic {\bf A}-free open sets is again a basic {\bf A}-free open set, those sets define a topology on $\Xi$, called the {\bf A-free topology}.

\begin{example}
\upshape Let $\V$ be the vector space $\mathbb{C}^d$ and $\Xi=\V_{nc}:=\coprod_{n=1}^{\infty}(\mathbb{C}^d)^{n\times n}$. We identify $(\mathbb{C}^d)^{n\times n}$ with $\mathbb{M}_n^d:=(\mathbb{C}^{n\times n})^d$ ($d$-tuples of $n\times n$ complex matrices), and hence we may view $\Xi$ as $\mathbb{M}^d:=\coprod_{n=1}^{\infty}\mathbb{M}_n^d$. We denote by $\mathcal{P}_d$ the algebra of all free polynomials in $d$-variables (A free polynomial, also called a noncommutative polynomial, in $d$-variables is a finite linear combination of words), and it induces a topology on $\mathbb{M}^d$ with letting ${\bf A}=\mathcal{P}_d$. This topology is simply called the {\bf free topology}.
\end{example}

So far, we have not yet introduced any kind of assumption for nc functions to be ``holomorphic". Here is such an assumption.

\begin{definition}
Let $\Omega$ be an {\bf A}-free open subset (may not be an nc subset) of a full nc subset $\Xi$ of $\V_{nc}$. A graded function $f:\Omega\rightarrow\mathbb{M}^1$ is {\bf A-free holomorphic} if for any $x\in\Omega$ there exists a basic {\bf A}-free open set $\mathbb{D}_Q$ such that it contains $x$ and $f$ gives a bounded nc function on $\mathbb{D}_Q$. 
\end{definition}

For nc functions, local boundedness and holomorphy are closely related (See e.g., \cite[Theorem 12.17]{AMY}, or \cite[Theorem 7.2]{KVV}). In fact, we remark that an {\bf A}-free holomorphic function is automatically holomorphic in the classical sense. Namely, the above definition is indeed a requirement for nc functions to be ``holomorphic".

The purpose of this note is to refine the following result when {\bf A} is a unital algebra (e.g., ${\bf A}=\mathcal{P}_d$). 

\begin{theorem}$($nc Oka-Weil Theorem, \cite[Theorem 3.3]{BMV}$)$\label{BMVOka}
Let f be an {\bf A}-free holomorphic function on an {\bf A}-free open subset
$\Omega$. Suppose that $K$ is an nc subset of $\Omega$ which is compact in the {\bf A}-free topology. Then there exists a sequence 
$\{p_n\}_{n=1}^{\infty}$ of functions in {\bf A} such that $p_n$ converges to $f$ uniformly on $K$.
\end{theorem}


\begin{remark}
\upshape Agler and McCarthy \cite[Theorem 9.7]{AMg} proved Theorem \ref{BMVOka} when ${\bf A}=\mathcal{P}_d$. They dealt with polynomially convex compact sets; see Theorem \ref{freecpt} at this point.
\end{remark}


\section{Known facts}

\subsection{A Characterization of {\bf A}-free compact nc sets}

In the rest of this note, we will assume that {\bf A} contains the identity function on $\Xi$, which is the nc function defined by $x\in\Xi_n\mapsto I_n\in\mathbb{M}_n^1$. 

\begin{definition}

An element $y\in\Xi_n$ is an {\bf A}-dilation of another $x\in\Xi_m$ if there is a positive integer $k\in\mathbb{N}$ and an isometry $V:\mathbb{C}^n\rightarrow\mathbb{C}^{km}$ such that for all p in {\bf A},
\begin{equation*}
p(y)=V^*p(x)^{(k)}V.
\end{equation*}
Here, $p(x)^{(k)}$ is the element of $\mathbb{M}_{km}^1$ obtained by taking the direct sum of k copies of $p(x)$.
\end{definition}


\begin{definition}
For an $x\in\Xi$, the {\bf A-dilation hull} $DH_{{\bf A}}(x)$ of $x$ is defined by
\begin{equation*}
DH_{{\bf A}}(x):=\left\{
y\in\Xi\;\middle|\;\mbox{y {\bf A}-dilates to x}\;\right\}.
\end{equation*}
\end{definition}


Augat, Balasubramanian and McCullough \cite{ABM} gave a characterization of free compact nc subsets in $\mathbb{M}^d$. The following result is its slight generalization. 


\begin{theorem}\label{freecpt}

An nc set $K\subset\Xi$ is {\bf A}-free compact if and only if there exists a $x\in K$ such that $K\subset DH_{{\bf A}}(x)$. In particular, a free compact set $K\subset\mathbb{M}^d$ agrees with $DH_{{\bf \mathcal{P}}_d}(x)$ for some $x\in K$ if and only if $K$ is polynomially convex (see \cite[Definition 9.5]{AMg}).
\end{theorem}
\begin{proof}
This is proved in the same way as the proof of Theorem 1.1 of \cite{ABM} because only the following were used there: the nc property of subsets (i.e., closed under the direct sums), and two general results, Arveson's extension theorem (\cite[Corollary 7.6]{Pau}) and Choi's Theorem (\cite[Proposition 4.7]{Pau}). To apply Corollary 7.6 of \cite{Pau}, we have to assume that {\bf A} contains the identity function.
\end{proof}


\subsection{Realization formula for the nc Schur-Agler class}

An {\bf A}-free holomorphic function is locally an nc function that is bounded on some basic {\bf A}-free open set $\mathbb{D}_Q$. Therefore, it is natural to study $H^{\infty}(\mathbb{D}_Q)$, the bounded nc functions on $\mathbb{D}_Q$. The {\bf nc Schur-Agler class}, $\mathcal{SA}(\mathbb{D}_Q)$ defined by
\begin{equation*}
\mathcal{SA}(\mathbb{D}_Q):=\left\{f:\mathbb{D}_Q\rightarrow\mathbb{M}^1\;\middle|\;\mbox{$f$ is nc and} \displaystyle\sup_{z\in\mathbb{D}_Q}\|f(z)\|\le 1\;\right\}
\end{equation*}
is studied by Agler and McCarthy  \cite{AMg} in the matrix case and by Ball, Marx and Vinnikov \cite{BMV} in general. They showed that each element of the nc Schur-Agler class admits a realization formula.


\begin{theorem}{$($\cite[Corollary 3.4]{BMV}$)$}\label{realization}
Let $\mathbb{D}_Q$ be a basic {\bf A}-free open set of $\Xi$ (the size of the matrix $Q$ is $s\times r$), and let $f$ be a graded function from $\mathbb{D}_Q$ into $\mathbb{M}^1$. Then the following conditions are equivalent:
\begin{enumerate}
\item $f\in\mathcal{SA}(\mathbb{D}_Q)$.
\item There exist an auxiliary Hilbert space $\mathcal{X}$ and a unitary operator
\begin{equation*}
U=
\begin{bmatrix}
A&B\\
C&D
\end{bmatrix}:
\begin{bmatrix}
\mathbb{C}^s\otimes\mathcal{X}\\
\mathbb{C}
\end{bmatrix}\rightarrow
\begin{bmatrix}
\mathbb{C}^r\otimes\mathcal{X}\\
\mathbb{C}
\end{bmatrix}
\end{equation*}
such that for all $z\in(\mathbb{D}_Q)_n$, 
\begin{equation*}
f(z)=
\begin{matrix}
D\\
\otimes\\
I_n
\end{matrix}
+
\begin{matrix}
C\\
\otimes\\
I_n
\end{matrix}\left(
\begin{matrix}
I_{\mathcal{X}}\\
\otimes\\
I_{n\times s}
\end{matrix}
-
\begin{matrix}
I_{\mathcal{X}}\\
\otimes\\
Q(z)
\end{matrix}
\begin{matrix}
A\\
\otimes\\
I_n
\end{matrix}
\right)^{-1}
\begin{matrix}
I_{\mathcal{X}}\\
\otimes\\
Q(z)
\end{matrix}
\begin{matrix}
B\\
\otimes\\
I_n
\end{matrix}\;\;.
\end{equation*}
\end{enumerate}
\end{theorem}

Here is a simple but important observation that will crucially be used later.


\begin{cor}\label{pointinterpolation}
Let $\mathbb{D}_Q$ be a basic {\bf A}-free open set of $\Xi$ and let
 $f\in H^{\infty}(\mathbb{D}_Q)$, that is, $f$ is a bounded nc function from $\mathbb{D}_Q$ into $\mathbb{M}^1$. Then, for any 
 $z\in\mathbb{D}_Q$, there exists a $p\in {\bf A}$ such that $f(z)=p(z)$.
\end{cor}
\begin{proof}

We may and do assume that $f\in\mathcal{SA}(\mathbb{D}_Q)$. Since $\|Q(z)\|<1$ for $z\in(\mathbb{D}_Q)_n$, we have the following realization for $f$ as an infinite series:
\begin{equation*}
f(z)=
\begin{matrix}
D\\
\otimes\\
I_n
\end{matrix}
+\sum_{k=0}^{\infty}
\begin{matrix}
C\\
\otimes\\
I_n
\end{matrix}\left(
\begin{matrix}
I_{\mathcal{X}}\\
\otimes\\
Q(z)
\end{matrix}
\begin{matrix}
A\\
\otimes\\
I_n
\end{matrix}
\right)^k
\begin{matrix}
I_{\mathcal{X}}\\
\otimes\\
Q(z)
\end{matrix}
\begin{matrix}
B\\
\otimes\\
I_n
\end{matrix}\;\;.
\end{equation*}
Since the matrix entries of $Q$ are all in the unital algebra {\bf A}, it follows that each partial sum $p_n$ of the infinite series falls in {\bf A}. Since $\{p(z)\in\mathbb{M}_n^1\;|\;p\in{\bf A}\}$ is a subspace of $\mathbb{M}_n^1$, it is automatically closed in the norm topology and there exists a $p\in {\bf A}$ such that $p(z)=\displaystyle\lim_{n\to\infty} p_n(z)=f(z)$.
\end{proof}
\begin{remark}
\upshape Agler, McCarthy and Young \cite[Theorem 14.4]{AMY} proved this result by using the Hahn-Banach theorem in the matrix case.
\end{remark}


\section{Main results}

We will introduce an {\bf A}-Zariski closure of a singleton. For a $\lambda\in\Xi$, we define the ideal
\begin{equation*}
I_{\lambda}:=\{p\in{\bf A}\;|\;p(\lambda)=0\}
\end{equation*}
and the {\bf A}-Zariski closure of $\lambda$ by
\begin{equation*}
V_{\lambda}:=\{x\in\Xi\;|\;p(x)=0\;\mbox{whenever}\;p\in I_{\lambda}\}.
\end{equation*}

Here is the main observation of this note.


\begin{theorem}\label{oka}
Let $f$ be a bounded graded function from a basic {\bf A}-free open set $\mathbb{D}_Q$ into $\mathbb{M}^1$. Then the following conditions are equivalent:
\begin{enumerate}
\item $f$ is nc, that is, $f\in H^{\infty}(\mathbb{D}_Q)$.
\item For any $\lambda\in\mathbb{D}_Q$, there exists a $p\in{\bf A}$ such that $f$ coincides with $p$ on $V_{\lambda}\cap\mathbb{D}_Q$.
\item For each {\bf A}-free compact nc subset $K$ of $\mathbb{D}_Q$, there exists a $p\in{\bf A}$ such that $f$ coincides with $p$ on $K$.
\end{enumerate}
\end{theorem}
\begin{proof}
(1)$\Rightarrow$(2). For any $\lambda\in\mathbb{D}_Q$, Corollary \ref{pointinterpolation} shows that there exists a $p\in{\bf A}$ such that $f(\lambda)=p(\lambda)$. If $x\in V_{\lambda}\cap\mathbb{D}_Q$, then $x\oplus\lambda$ is in $\mathbb{D}_Q$, and therefore, there is a $q\in{\bf A}$ such that $f(x\oplus\lambda)=q(x\oplus\lambda)$. Since $f$ and $q$ respect direct sums, $f(x)=q(x)$ and $f(\lambda)=q(\lambda)$ hold. Therefore, $p-q$ is in $I_{\lambda}$ and $f(x)=q(x)=p(x)$ due to $x\in V_{\lambda}$.\\
(2)$\Rightarrow$(3). By Theorem \ref{freecpt}, each {\bf A}-free compact nc subset $K$ of $\mathbb{D}_Q$ must sit in the {\bf A}-dilation hull $DH_{{\bf A}}(\lambda)$ of a $\lambda\in K$. By definition, $DH_{{\bf A}}(\lambda)$ sits in $V_{\lambda}$. Hence $K\subset V_{\lambda}\cap\mathbb{D}_Q$. Thus item (2) implies item (3).\\
(3)$\Rightarrow$(1). It is sufficient to prove that $f$ respects intertwinings. Let $x, y\in\mathbb{D}_Q$ and let $\alpha$ be a matrix with $\alpha x=y\alpha$. Obviously, $x\oplus y$ is in $\mathbb{D}_Q$. Since all elements of {\bf A} respect direct sums, we have
\begin{equation*}
p(x)=
\begin{bmatrix}
I&0
\end{bmatrix}
\begin{bmatrix}
p(x)&0\\
0&p(y)
\end{bmatrix}
\begin{bmatrix}
I\\
0
\end{bmatrix}
=
\begin{bmatrix}
I&0
\end{bmatrix}p\left(
\begin{bmatrix}
x&0\\
0&y
\end{bmatrix}\right)
\begin{bmatrix}
I\\
0
\end{bmatrix}
\end{equation*}
for all $p\in{\bf A}$. This implies $x\in DH_{{\bf A}}(x\oplus y)$. 
By the same calculation, $y$ also belongs to $DH_{{\bf A}}(x\oplus y)$. Since
 $DH_{{\bf A}}(x\oplus y)$ is an nc {\bf A}-free compact subset of 
$\mathbb{D}_Q$, there exists a $p\in{\bf A}$ such that $f(x)=p(x)$ and $f(y)=p(y)$. As $p$ respects intertwinings, we obtain that
\begin{equation*}
\alpha f(x)=\alpha p(x)=p(y)\alpha=f(y)\alpha.
\end{equation*}
Therefore, $f$ respects intertwinings.
\end{proof}


Here is a simple corollary, which is nothing but a refined nc Oka-Weil theorem.

\begin{cor}\label{roka}
Let $f$ be an {\bf A}-free holomorphic function on an {\bf A}-free open set $\Omega$. If $K$ is an {\bf A}-free compact nc subset of $\Omega$, then there exists a $p\in{\bf A}$ such that $f$ coincides with $p$ on $K$. In particular, any free holomorphic function coincides with some free polynomial on each free compact nc subset.
\end{cor}
\begin{proof}
By Theorem \ref{freecpt}, we may assume that $K=DH_{{\bf A}}(x)$ for some $x\in\Omega$. Since $f$ is {\bf A}-free holomorphic, there exists a basic {\bf A}-free open subset
 $\mathbb{D}_Q\subset\Omega$ that contains $x$, and moreover, $f\in H^{\infty}(\mathbb{D}_Q)$. By the definition of an {\bf A}-dilation hull, it is easy to see that $DH_{{\bf A}}(x)$ sits in $\mathbb{D}_Q$. Hence, Theorem \ref{oka} implies that we can find an element of {\bf A} that agrees with $f$ on $DH_{{\bf A}}(x)$. 
\end{proof}


We have seen that every element of the nc Schur-Agler class agrees with a function in {\bf A} on each {\bf A}-free compact nc subset (Theorem \ref{oka}). This was a consequence from Theorem \ref{freecpt}. Corollary \ref{roka} suggests that free compact nc subsets should be regarded as nc counterparts of finite subsets (rather than compact subsets) in the complex plane. In fact, any function on the complex plane coincides with a polynomial on a finite subset. We will briefly explain a cryptic phenomenon of free holomorphic functions from this viewpoint. Pascoe \cite{PasE} proved that if $d>1$, then there is an entire free holomorphic function (hence it is free continuous due to \cite[Proposition 3.8]{AMYm}) which is unbounded on the row ball. This phenomenon never occurs in the classical setting. Since the closed row ball is nc, our observation here says that the closed row ball is not free compact. In this respect,  Pascoe already pointed out the same kind of flaw in the free topology in \cite{PasE}. Moreover, Agler, McCarthy and Young pointed out another flaw, that is, the free topology is not Hausdorff \cite[Corollary 7.6 and Proposition 7.13]{AMYm}. 

Recall that any compact subsets in the finite topology must be compact in the free topology and also that the finite topology is, without doubt, a natural one because it can be regarded as an nc analogue of the usual Euclidean topology. Thus due to the above observation showing that any free compact nc subsets are too small in some sense, it seems natural to discuss the Oka-Weil type polynomial approximation on free compact subsets that are not necessarily nc. Here we examine Pascoe's observation (See \cite[Page 20]{Pas}) for this problem. By Theorem \ref{realization}, if $f\in\mathcal{SA}(\mathbb{D}_Q)$, then $f$ is realized as
\begin{equation*}
f(z)=
\begin{matrix}
D\\
\otimes\\
I_n
\end{matrix}
+
\begin{matrix}
C\\
\otimes\\
I_n
\end{matrix}\left(
\begin{matrix}
I_{\mathcal{X}}\\
\otimes\\
I_{n\times s}
\end{matrix}
-
\begin{matrix}
I_{\mathcal{X}}\\
\otimes\\
Q(z)
\end{matrix}
\begin{matrix}
A\\
\otimes\\
I_n
\end{matrix}
\right)^{-1}
\begin{matrix}
I_{\mathcal{X}}\\
\otimes\\
Q(z)
\end{matrix}
\begin{matrix}
B\\
\otimes\\
I_n
\end{matrix}\;\;.
\end{equation*}
We then define elements of {\bf A} by
\begin{equation*}
f_{N,r}(z)=
\begin{matrix}
D\\
\otimes\\
I_n
\end{matrix}
+\sum_{k=0}^N
\begin{matrix}
C\\
\otimes\\
I_n
\end{matrix}\left(
\begin{matrix}
I_{\mathcal{X}}\\
\otimes\\
rQ(z)
\end{matrix}
\begin{matrix}
A\\
\otimes\\
I_n
\end{matrix}
\right)^k
\begin{matrix}
I_{\mathcal{X}}\\
\otimes\\
rQ(z)
\end{matrix}
\begin{matrix}
B\\
\otimes\\
I_n
\end{matrix}\;\;
\end{equation*}
where $0<r<1$. Note that for each $0<r<1$, there is an $N_0$ such that if $N\ge N_0$ then $\displaystyle\sup_{z\in\mathbb{D}_Q}\|rf_{N,r}(z)\|\le 1$. In this way, we can prove the following fact.


\begin{theorem}
A graded function $f$ from a basic {\bf A}-free open set $\mathbb{D}_Q$ into $\mathbb{M}^1$ belongs to $\mathcal{SA}(\mathbb{D}_Q)$ if and only if there exists a sequence 
$\{p_n\}_{n=1}^{\infty}$ in {\bf A} such that $p_n$ converges to
 $f$ uniformly on each (not necessarily nc) {\bf A}-free compact subset of $\mathbb{D}_Q$, and the norm of $p_n$ is uniformly less than one.
 \end{theorem}


In closing we would like to present the following natural question. If it was affirmative, one would be able to find the $p\in{\bf A}$ in Theorem \ref{oka} with the additional property $\displaystyle\sup_{z\in\mathbb{D}_Q}\|p(z)\|\le\displaystyle\sup_{z\in\mathbb{D}_Q}\|f(z)\|$.

\begin{question}
Let $f\in\mathcal{SA}(\mathbb{D}_Q)$. For each $\lambda\in\mathbb{D}_Q$, is there a $p\in{\bf A}$ such that $f(\lambda)=p(\lambda)$, and $\displaystyle\sup_{z\in\mathbb{D}_Q}\|p(z)\|\le 1$$\mathord{?}$
\end{question}


\section*{Acknowledgments}

The author gratefully acknowledges his supervisor Professor Yoshimichi Ueda for pointing out that Theorem \ref{oka} and its corollary seem to suggest that any free compact nc subsets behave like finite  subsets. The author also thanks Professor James Eldred Pascoe and the referee for making him aware of several related works including \cite{PasE}.

\end{document}